\documentclass[11pt]{article}

\usepackage[utf8]{inputenc}
\usepackage[T1]{fontenc}

\usepackage{amsmath,amssymb,amsthm,amsfonts}
\usepackage{mathtools}
\usepackage{graphicx}
\usepackage{enumitem}
\usepackage{mathrsfs}
\usepackage{bm}
\usepackage{hyperref}
\usepackage[numbers]{natbib}

\usepackage[a4paper,margin=1in]{geometry}

\usepackage{fancyhdr}

\newtheorem{theorem}{Theorem}[section]
\newtheorem{lemma}[theorem]{Lemma}

\newtheorem{proposition}[theorem]{Proposition}

\theoremstyle{definition}
\newtheorem{definition}[theorem]{Definition}
\newtheorem{example}[theorem]{Example}

\theoremstyle{remark}
\newtheorem{remark}[theorem]{Remark}

\title{
\vspace{-1cm}
\hrule
\vspace{0.4cm}
\Large\bfseries
Partial Extended b-Metric Space and Some Fixed Point Theorem
\vspace{0.4cm}
\hrule
}

\author{
Muhamad Abdillah Ahen$^{1}$ \and
Ivan Hadinata$^{2}$ \and
Raudhatul Mufizah$^{3}$\\[0.5cm]
$^{1}$Department of Mathematics, Universitas Hasanuddin\\
\texttt{ahenma23h@student.unhas.ac.id}\\
$^{2}$Department of Mathematics, Universitas Gadjah Mada\\
\texttt{ivanhadinata2005@mail.ugm.ac.id}\\
$^{3}$Department of Mathematics, Universitas Hasanuddin\\
\texttt{mufizahr@gmail.com}\\
}

\date{\today}

\begin{document}

\maketitle

\begin{abstract}
    In this paper, we introduce the concept of partial extended b-metric spaces (PEBMS) as a unification and generalization of extended b-metric spaces and partial b-metric spaces. This new structure incorporates a point-dependent control function together with the possibility of non-zero self-distance, providing a more flexible framework for the study of generalized metric spaces. We establish several fundamental properties of PEBMS, including convergence, Cauchy sequences, and 0-completeness. By introducing the notion of 0-Cauchy sequences, we extend various results from extended b-metric spaces to the PEBMS setting. In particular, we prove fixed point theorems for contractive mappings and show the existence and uniqueness of fixed points under suitable conditions. Furthermore, we demonstrate that every extended b-metric space can be viewed as a special case of a PEBMS. As an application, we study the stability of discrete dynamical systems within this framework. The results presented here generalize and enrich existing theories in metric-type spaces and open new directions for further research.
\end{abstract}


\section{Introduction}
The classical metric space and Banach's fixed-point theorem (1922) have become an important foundation in mathematical analysis \cite{Karahan2018PARTIALBP,kamran2017,Chen2020}. However,  classical metric spaces are often insufficient for modeling more complex problems. This limitation has motivated the development of various generalizations, such as the $b$-metric space \cite{Bakhtin1989,Czerwik1993}, which modify the triangle inequality with a constant $s \geq 1$, and the extended $b$ -metric space \cite{kamran2017}, which replace the constant $s$ with a function $\theta (x,z) \geq 1$ that depends on the points involved. Since their introduction, many results of fixed-point theory have been developed within this framework, including generalizations of the Banach, Kannan, and Chatterjea fixed-point theorems \cite{Dung2017,Bhattacharjee2025}, and studies on multivalued operators and Suzuki-type contractions \cite{Boriceanu2010,Hussain2012}. Conversely, Matthews \cite{Matthews1994} introduced the concept of partial metric spaces, which permits the distance from a point to itself to be nonzero. This was later combined with $b$-metric spaces by Shukla \cite{Shukla2014}, to form partial $b$-metric spaces.

Although these generalizations have produced many fixed-point theorems, most research still addresses each space structure separately. Some studies have attempted to combine two concepts, such as partial $b$-metric spaces, which only merge two concepts but are limited to certain types of contractions. In addition, the results in extended $b$-metric spaces generally focus on specific classes of contractions and do not simultaneously accommodate various types of contractions within a unified space structure. These limitations reveal a research gap in developing a space capable of integrating the characteristics of partial and extended $b$-metric structures while supporting various types of contraction mappings.

Based on this background, this study is inspired to develop the concept of Partial Extended b-metric Space as a generalization that combines the flexible distance structures of partial metric and extended b-metric spaces. Through this framework, it is expected that new fixed-point theorems for several types of contraction mappings can be obtained simultaneously, thus offering a more comprehensive contribution to the advancement of fixed-point theory.

\section{Some Consepts}
\label{sec:headings}

In this section, we recall several basics definitions and preliminary results related to extended $b$-metric spaces and partial $b$-metric spaces. In particular, we present some fundamental notions together with auxiliary lemmas and known fixed point results that will be used in the sequel. These concepts form the theoretical foundation for the main results on fixed points in partial extended b-metric space. 
\vspace{5mm}
\begin{definition} \cite{kamran2017} \label{def1}
    Let $X$ be a non empty set and $\theta : X \times X \to [1, \infty)$. A function $d_\theta :X\times X \to [0,\infty) $ is called extended b-metric (EBM) if for all $x,y,z\in X$ it satisfies: 
    \begin{enumerate}
        \item $d_\theta (x,y)=0 \quad \iff \quad x=y$
        \item $d_\theta (x,y) = d_\theta (y,x)$
        \item $d_\theta (x,z) \leq \theta(x,z) [d_\theta (x,y)+d_\theta (y,z)]$
    \end{enumerate}
    The pair ($X$,$d_\theta$) is called extended b-metric space (EBMS)
\end{definition}

\vspace{2mm}
\begin{remark}
    if $\theta(x,y)=s$ for $s \geq 1$ we obtain the definition of b-metric space \cite{Bakhtin1989,Czerwik1993}
\end{remark}
\vspace{2mm}
\begin{example}
    Let $X=\{2,3,4 \}$, $\theta:X\times X \to [1, \infty)$, and $d_\theta:X\times X \to [0,\infty)$ as follows: 
    $$\theta(x,y)=1+x+y $$ 
    $$d_\theta (x,y) = \begin{cases} 0 \quad if \quad x=y \\
    20 \quad if \quad  x \neq y      
    \end{cases}
    $$ show that The pair ($X$,$d_\theta$) is EBMS

\begin{proof}
    
from definition ~\ref{def1} (1) and (2) are done by replacing $x=y$. For (3) we have 
\begin{align*}
    d_\theta (2,3) &= 20 \\
    &\leq \theta (2,3)[d_\theta (2,4)+d_\theta (4,3)] \\
    &= 240
\end{align*}
\begin{align*}
    d_\theta (2,4) &= 20 \\
    &\leq \theta (2,4)[d_\theta (2,3)+d_\theta (3,4)] \\
    &= 280
\end{align*}
in a similar way to $d_\theta (3,4)$, Hence for all $x,y,z \in X$ $$d_\theta (x,y)\leq \theta (x,y)[d_\theta (x,z)+d_\theta (z,y)]$$
Then we have the pair of $(X,d_\theta)$ is a EBMS
\end{proof} 
\end{example}

\vspace{5mm}
\begin{definition}\cite{agarwal2018banach}
    Let $(X, d)$ be a metric space. The mapping $T : X \to X$ is said to be Lipschitzian if there exist a constant $k>0$ (called Lipschitz constant) such that $$d(Tx, Ty) \leq kd(x,y), \quad (x,y) \in X \times X.$$
    A Lipschitzian mapping with a Lipschitz constant $k < 1$ is called contraction.
\end{definition}

\vspace{2mm}
\begin{definition}\cite{kamran2017}
    Let $(X, d_\theta)$ be an extended $b$-metric space. A sequence $\{x_n\}$ in $X$ is said to be:
    \begin{enumerate}
        \item Cauchy if and only if $d(x_n, x_m) \to 0$as $n,m \to \infty$;
        \item Convergent if and only if there exist $x\in X$ such that $d(x_n, x) \to 0$ as $n \to \infty$ and we write $\lim_{n \to \infty} x_n = x$;
        \item The $b$-metric space $(X, d)$ is complete if for every Cauchy sequence is convergent.
    \end{enumerate}
\end{definition}

\vspace{2mm}
\begin{definition}\cite{Akkouchi2023}
    Let $X$ be a non empty set and let $T$ be a self-mapping of $X$. Then, for every $x \in X$, the set $O_{T}(x) := \{x, T_{x}, T^2_{x}, T^3_{x}, ...$ is called the orbit of T at $x \in X$.
\end{definition}

\vspace{2mm}
\begin{lemma}\label{lem1} \cite{kamran2017}
    Let ($X$,$d_\theta$) be an extended b-metric space. If $d_\theta$ is contiuous, then every convergent sequence has a unique limit
\end{lemma}

\vspace{2mm}
\begin{theorem}\label{Teo1} \cite{kamran2017}
    Let ($X$,$d_\theta$) be a complete extended b-metric space such that $p$ is a continuous functional. Let a self mapping $T:X\to X$ satisfy: $$p(Tx,Ty) \leq k\:p(x,y) \quad, \forall x,y\in X$$
    where $k\in [0,1)$ be such that for each $x_0 \in X$, $\lim_{n,m \to \infty} \theta (x_n,x_m) < \frac{1}{k}$, where $T^nx_0=x_n \quad ,n=1,2,3,...$. Then $T$ has precisely one fixed point $u$. Moreover for each $y\in X$, $T^ny \to u$
\end{theorem}

\vspace{2mm}
\begin{definition}\cite{Karahan2018PARTIALBP}\label{defPBM}
    Let $X$ be a non-empty set and coeficient $s\geq 1$. A function $p:X \times X \to [0,\infty )$ is called partial b-metric (PBM) on $X$ if for all $x,y,z \in X$ it satisfies: 
\begin{enumerate}
    \item $p(x,x)=p(x,y)=p(y,y)$ if and only if $x=y$ 
    \item $p(x,x)\leq p(x,y)$
    \item $p(x,y)=p(y,x)$
    \item $p(x,z)\leq s[p(x,y)+p(y,z)]-p(y,y)$
\end{enumerate}
    The triplet of $(X,p,s)$ is called partial b-metric space (PBMS)
\end{definition}

\begin{remark} (\cite{Shukla2014}, Remarks 1 and 2)
    \begin{enumerate}
        \item In a partial $b$-metric space $(X,b)$, condition $b(x,y)=0$ implies that $x=y$; however, the converse implication does not necessarily hold.
        
        \item A partial metric space can be viewed as a special case of a partial $b$-metric space with coefficient $s=1$. Similarly, every $b$-metric space is a partial $b$-metric space with the same coefficient and zero self-distance. Nevertheless, the converse statements are not necessarily true.
    \end{enumerate}
\end{remark}

\vspace{2mm}
\begin{definition}\cite{Shukla2014}
    Let $(X, p)$ be a partial $b$-metric space with coefficient $s$. Let $\{x_n\}$ be any sequence in $X$ and $x \in X$. Then:
    \begin{enumerate}
        \item The sequence $\{x_n\}$ is said to be convergent with respect to $\tau_b$ and converges to $x$, if $\lim_{n \to \infty} p(x_n, x) = p(x,x)$.
        
        \item Sequence $\{x_n\}$ is said to be Cauchy sequence in $(X, p)$ if $\lim_{n,m \to \infty}p(x_n,x_m)$ exist and is finite.
        \item $(X, p)$ is said to be a complete partial $b$-metric space if for every Cauchy sequence $\{x_n\}$ in $X$ there exist $x \in X$ such that $$\lim_{n,m \to \infty} p(x_n, x_m) = \lim_{n \to \infty} p(x_n, x) = p(x,x)$$
    \end{enumerate}
\end{definition}

\vspace{2mm}
\begin{definition}\cite{Dung2017}
   Let $(X,p)$ be a partial $b$-metric space.

    \begin{enumerate}
        \item A sequence $\{x_n\}$ is called a $0$-Cauchy sequence if 
        $$\lim_{n,m \to \infty} p(x_n,x_m) = 0$$.
        
        \item $(X,p)$ is called $0$-complete if for each $0$-Cauchy sequence 
        $\{x_n\}$ in $X$, there exists $x \in X$ such that
        $$\lim_{n,m \to \infty} p(x_n,x_m) = \lim_{n \to \infty} p(x_n,x) = p(x,x) = 0$$.
    \end{enumerate} 
\end{definition}

\begin{example}
    Let $X = (0, \infty)$ and let $b>1$ be a fixed real constant. Define a mapping $b : X \times X \to X$ by $$p(x,y) =[\max\{x,y\}]^b + |x-y|^b, \quad \forall x,y \in X.$$ Show that $(X, p)$ is partial $b$-metric space with coefficient $s = 2^b$
\end{example} 

\section{Result and Discussion}

\vspace{.5cc}
Recently, the concept of partial $b$-metric space was introduced by Shukla in \cite{Shukla2014} and further developed by Van Dung and Hang in \cite{Dung2017}. Motivated by these developments, the partial extended $b$-metric is introduced as a natural of the existing structure.
\vspace{2mm}
\begin{definition}
    Let $X$ be a non-empty set and $\theta:X \times X \to [1,\infty )$. A function $p_{\theta}:X \times X \to [0,\infty )$ is called partial extended b-metric (PEBM) on $X$ if for all $x,y,z \in X$ it satisfies: 
\begin{enumerate}
    \item $p_\theta(x,x)=p_\theta(x,y)=p_\theta(y,y)$ if and only if $x=y$ 
    \item $p_\theta(x,x)\leq p_\theta(x,y)$
    \item $p_\theta(x,y)=p_\theta(y,x)$
    \item $p_\theta(x,z)\leq \theta (x,z)[p_\theta(x,y)+p_\theta(y,z)]-p_\theta(y,y)$
\end{enumerate}
    The triplet of $(X,p_\theta,\theta)$ is called partial extended b-metric space (PEBMS)
\end{definition}
\vspace{5mm}
\begin{example}
    Let $X = [0,1]$. Define $\theta : X \times X \to [1, \infty)$ by $\theta(x,y) = x+y+1, \forall x, y \in X$. Define $p_{\theta} : X \times X \to \infty [0, \infty$ by $p_{\theta}(x,y) = |x-y| + x$. Then $(X, p_{\theta}, \theta)$ is a partial extended $b$-metric space.
\end{example}
\vspace{5mm}
\begin{remark}(Remarks 1 and 2)
    \begin{enumerate}
        \item The notion of a partial extended $b$-metric extends the concept of an extended $b$-metric by allowing non–zero self-distance. In particular, while an extended $b$-metric satisfies $d_\theta(x,x)=0$ for all $x\in X$, a partial extended $b$-metric admits $p_\theta(x,x)\ge 0$ and modifies the triangle inequality by the term $-p_\theta(y,y)$. Hence, every extended $b$-metric space can be regarded as a special case of a partial extended $b$-metric space when the self-distance is zero.
        \item The concept of a partial extended $b$-metric space (PEBMS) used in this paper differs from the notion of an extended partial $b$-metric space (or partial $p$-metric space). The function of partial extended $b$-metric spcae $\theta : X \times X \to [1, \infty)$ acts as control acts as a control function in the generalized triangle inequality, whereas in the extended partial $b$-metric spcae (or partial p-metric space), the structure is governed by a function $\Omega : [0, \infty) \to [0, \infty)$ applied to the metric expression. Hence, these two notions are not equivalent and lead to different analytical frameworks.
    \end{enumerate} 
\end{remark}

\vspace{.5cc}

This definition is inspired by the notion of convergent sequences in partial $b$-metric spaces \cite{Shukla2014}, and is formulated by adapting and extending the underlying ideas to suit the present context.
\vspace{2mm}
\begin{definition}
    Let ($X$,$p_\theta$,$\theta$) be a partial extended b-metric space 
    \begin{enumerate}
        \item A sequence $\{x_n\}$ is called convergent to $x$ in $X$, written $\lim_{n\to \infty} x_n =x$ if 
        $$\lim_{n\to \infty} p_\theta(x_n,x)=p_\theta(x,x)$$
        \item The sequence $\{x_n\}$ is said to be Cauchy sequence in $(X, p_\theta, \theta)$ if $\lim_{n,m \to \infty}p_\theta(x_n,x_m)$ exist and is finite.
        \item $(X, p_\theta, \theta)$ is said to be a complete partial $b$-metric space if for every Cauchy sequence $\{x_n\}$ in $X$ there exist $x \in X$ such that $$\lim_{n,m \to \infty} p_\theta(x_n, x_m) = \lim_{n \to \infty} p_\theta(x_n, x) = p_\theta(x,x)$$
    \end{enumerate}
\end{definition}
\begin{example}
    Let $X = [0,1]$. Define $\theta : X \times X \to [1, \infty)$ by $\theta(x,y) = 1 + x + y$. Define $p_{\theta} : X \times X \to [0, \infty]$ by $p_{\theta} = \max\{x,y\}$. Consider the sequence $\{x_n\} \subset X$ with $x_n = \frac{1}{n}$. Then $x_n \to 0$ in $(X, p_{\theta}, \theta)$, the sequence $\{x_n\}$ is Cauchy and $(X, p_{\theta}, \theta)$ is a complete partial extended $b$-metric space.
\end{example}
\vspace{0.2mm}

The definition below is inspired by the concept of a $0$-Cauchy sequence introduced by Van Dung \cite{Dung2017}, and has been adapted to fit the context discussed in this work.
\begin{definition}
Let $(X,p_\theta,\theta)$ be a partial extended $b$-metric space with the functional $\theta$.
\begin{enumerate}
    \item A sequence $\{x_n\}$ is called a $0$-Cauchy sequence if 
        $$\lim_{n,m \to \infty} p_\theta(x_n,x_m) = 0$$.
        
    \item $(X,p_\theta,\theta)$ is called $0$-complete if for each $0$-Cauchy sequence 
        $\{x_n\}$ in $X$, there exists $x \in X$ such that
        $$\lim_{n,m \to \infty} p_\theta(x_n,x_m) = \lim_{n \to \infty} p_\theta(x_n,x) = p_\theta(x,x) = 0$$.
\end{enumerate}
\end{definition}
\begin{example}
    Let $X = \mathbb{R}^+ \cup \{0\}$ (i.e., $[0, \infty)$). Define $p_{\theta} : X \times X \to [0,\infty)$ by $p_{\theta}(x,y) = |x-y| + \min\{x,y\} $ and define $\theta : X \times X \to [1, \infty)$ by $\theta(x,y) = 1 + x+ y$. Consider the sequence $x_n = \frac{1}{n^2}$. Then, $p_{\theta} (x_n, x_m) \to 0$ as $n,m \to \infty$, so $\{x_n\}$ is a $0$-Cauchy sequence in $(X, p_{\theta}, \theta)$. However, $x_n \to 0 \notin X$, hence $(X, p_{\theta}, \theta)$ is not $0$-complete. 
\end{example}
\vspace{3mm}
\begin{proposition}\label{lemme1}
    Let $(X, p_\theta, \theta)$ be a partial extended b-metric space. Define a function $d_p : X \times X \to [0,\infty)$ by
    \[
    d_p(x,y) =
    \begin{cases}
    0, & x = y, \\
    p_\theta(x,y), & x \ne y.
    \end{cases}
    \]
    Then $d_p$ is an extended b-metric on $X$ with the same control function $\theta$.
\end{proposition}
\begin{proof}
    First, it is clear that $d_p : X \times X \to [0,\infty)$.
    (1) If $d_p(x,y)=0$, then either $x=y$ or $p_\theta(x,y)=0$. 
    From the properties of $p_\theta$, this implies $x=y$. Conversely, if $x=y$, then $d_p(x,y)=0$.
    (2) Symmetry follows directly from the symmetry of $p_\theta$, that is,
    \[
    d_p(x,y)=d_p(y,x).
    \]
    (3) For all $x,y,z \in X$, we consider two cases.
    If any two of $x,y,z$ coincide, the inequality holds trivially.
    If $x \ne y \ne z$, then
    \[
    d_p(x,z) = p_\theta(x,z)
    \le \theta(x,z)\big[p_\theta(x,y) + p_\theta(y,z)\big] - p_\theta(y,y).
    \]
    Since $p_\theta(y,y) \ge 0$, we obtain
    \[
    d_p(x,z) \le \theta(x,z)\big[d_p(x,y) + d_p(y,z)\big].
    \]
    Thus, $d_p$ satisfies the extended b-metric inequality. Therefore, $d_p$ is an extended b-metric on $X$.
\end{proof}

\vspace{.5cc}
The following theorem is inspired by theorem ~\ref{Teo1} and combines it with the $0$-complete partial b-metric.
\vspace{3mm}
\begin{theorem} \label{thm:theoreme-1}
Let $(X,p_\theta,\theta)$ be a $0$-complete partial extended $b$-metric space where $p_\theta$ is a continuous functional. Let $T:X\to X$ be a self-mapping satisfying
\begin{equation}\label{eq1}
p(Tx,Ty) \le k\,p(x,y), \quad \forall x,y \in X,
\end{equation}\label{eq1}
where $k \in [0,1)$. For each $x_0 \in X$, define $x_n = T^n x_0$. If
\begin{equation}\label{eq2}
\lim_{n,m \to \infty} \theta(x_n,x_m) < \frac{1}{k},
\end{equation}
then $T$ has a unique fixed point $u \in X$. Moreover, for each $y \in X$, $T^n y \to u$.
\end{theorem}\vspace{0.5mm}

\begin{proof}
Choose an arbitrary $x_0 \in X$ and define the iterative sequence $\{x_n\}$ by
\[
x_n = T^n x_0, \qquad n=0,1,2,\ldots
\]
Using the contractive condition
\[
p_\theta(Tx,Ty) \le k\,p_\theta(x,y),
\]
we obtain
\begin{equation}
    p_\theta(x_n,x_{n+1}) = p_\theta(Tx_{n-1},Tx_n)
    \le k\,p_\theta(x_{n-1},x_n)
    \le \cdots \le k^n p_\theta(x_0,x_1).
\end{equation}
Now, using the triangle inequality of the partial extended $b$-metric, for $m>n$ we have
    \begin{align}\label{eq5}
    p_\theta(x_n,x_m)
    &\le \theta(x_n,x_m)\big[p_\theta(x_n,x_{n+1}) + p_\theta(x_{n+1},x_m)\big]
    - p_\theta(x_{n+1},x_{n+1}) \nonumber \\
    &\le \theta(x_n,x_m)p_\theta(x_n,x_{n+1})
    + \theta(x_n,x_m)p_\theta(x_{n+1},x_m).
    \end{align}
Applying this inequality repeatedly and using \eqref{eq5}, we obtain
    \begin{align}
    p_\theta(x_n,x_m)
    &\le \theta(x_n,x_m)k^n p_\theta(x_0,x_1)
    + \theta(x_n,x_m)\theta(x_{n+1},x_m)k^{n+1} p_\theta(x_0,x_1) \nonumber \\
    &\quad + \cdots +
    \theta(x_n,x_m)\theta(x_{n+1},x_m)\cdots\theta(x_{m-1},x_m)
    k^{m-1} p_\theta(x_0,x_1).
    \end{align}
Hence,
    \[
    p_\theta(x_n,x_m) \le p_\theta(x_0,x_1)
    \sum_{i=n}^{m-1}
    \left(
    \prod_{j=n}^{i}\theta(x_j,x_m)
    \right) k^i .
    \]
Since $\lim_{n,m \to \infty}\theta(x_{n+1},x_m)k < 1,$
    the series
    \[
    \sum_{n=1}^{\infty} k^n \prod_{i=1}^{n} \theta(x_i,x_m)
    \]
    converges by the ratio test for each $m \in \mathbb{N}$. Let
    \[
    S=\sum_{n=1}^{\infty} k^n \prod_{i=1}^{n}\theta(x_i,x_m), 
    \qquad
    S_n=\sum_{j=1}^{n} k^j \prod_{i=1}^{j}\theta(x_i,x_m).
    \]
Thus for $m>n$, from the previous inequality we obtain
    \[
    p_\theta(x_n,x_m)
    \le
    p_\theta(x_0,x_1)\,[S_{m-1}-S_n].
    \]
Letting $n \to \infty$ we conclude that $\{x_n\}$ is a $0$-Cauchy sequence.Since $(X,p_\theta,\theta)$ is $0$-complete, there exists $\xi \in X$ such that
    \[
    x_n \to \xi .
    \]
Now we show that $\xi$ is a fixed point of $T$. Using the triangle inequality of the partial extended $b$-metric, we have
    \begin{align*}
    p_\theta(T\xi,\xi)
    &\le \theta(T\xi,\xi)\big[p_\theta(T\xi,x_n)+p_\theta(x_n,\xi)\big]
    - p_\theta(x_n,x_n) \\
    &\le \theta(T\xi,\xi)\big[p_\theta(T\xi,Tx_{n-1})+p_\theta(x_n,\xi)\big] \\
    &\le \theta(T\xi,\xi)\big[k\,p_\theta(\xi,x_{n-1})+p_\theta(x_n,\xi)\big].
    \end{align*}
Taking the limit as $n \to \infty$, we obtain
    \[
    p_\theta(T\xi,\xi)=0.
    \]
Hence $T\xi=\xi$, and therefore $\xi$ is a fixed point of $T$. Finally, the uni queness of the fixed point follows directly from the contractive condition \eqref{eq1}, since $k<1$.
\end{proof}

\vspace{.5cc}
Before presenting the following theorem, we note that its proof is based on a Kannan-type contraction condition \cite{Kannan1968}, adapted from the previous result.
\begin{theorem}
    Let $(X,p_\theta ,\theta)$ be a  $0$-complete partial extended b-metric space. Let $T : X \to X$ satisfies. 
    \begin{equation} \label{Kannan}
        p_\theta (Tx,Ty) \leq k \left[ p_\theta(x,Tx) + p_\theta(y,Ty) \right]
    \end{equation}
    with $\theta(Tx,x) < \frac{1}{k}$ for all $x,y \in X$, where the constant $k \in [0,\frac{1}{2})$. Then $T$ has a unique fixed point $u\in X$. Moreover $p_\theta(u,u)=0$ and $\lim_{n\to \infty} T^nx_0 =u $ for all $x_0\in X$.
\end{theorem}

\begin{proof}
    Let us consider $x_0\in X$ arbitrarily. Define sequence $(x_n)$ in $X$ with the initial $x_0$ and $x_n = Tx_{n-1}$, $n=0,1,2, \dots $. From inequality $p_\theta (Tx,Ty) \leq k \left[ p_\theta(x,Tx) + p_\theta(y,Ty) \right]$ 
    setting $x=x_{m-1}$ and $y=x_{n-1}$ we have 
    \begin{equation}\label{Theo2: ineq1}
        p_\theta (x_m,x_n) \leq k \left[ p_\theta(x_m,x_{m-1}) + p_\theta(x_n,x_{n-1}) \right] \quad , \forall m,n \in \mathbb{N}
    \end{equation}
    in a similar way setting $x=x_{n}$ and $y=x_{n-1}$ we have 
    \begin{equation*}\label{Theo2: ineq2}
         p_\theta (x_{n+1},x_{n}) \leq k \left[ p_\theta(x_{n+1},x_{n}) + p_\theta(x_n,x_{n-1}) \right] \quad , \forall n \in \mathbb{N}       
    \end{equation*}
    $$ (1-k) \: p_\theta (x_{n+1},x_{n})  \leq k\: p_\theta(x_n,x_{n-1})  \quad , \forall n \in \mathbb{N}$$
    $$ \: p_\theta (x_{n+1},x_{n})  \leq \frac{k}{1-k}\: p_\theta(x_n,x_{n-1})  \quad , \forall n \in \mathbb{N}$$
    $$ \: p_\theta (x_{n+1},x_{n})  \leq \frac{k}{1-k}\: p_\theta(x_n,x_{n-1}) \leq \left( \frac{k}{1-k}\right)^2 p_\theta(x_{n-1},x_{n-2}) \quad , \forall n \in \mathbb{N}$$
    By repeating the same recursive step, we obtain. 
    \begin{equation} \label{Theo2: ineq3}
        \: p_\theta (x_{n+1},x_{n})  \leq \left( \frac{k}{1-k}\right)^n\: p_\theta(x_1,x_{0})  \quad , \forall n \in \mathbb{N}
    \end{equation}
    Since $k \in [0, \frac{1}{2})$, then $\left( \frac{k}{1-k}\right) \in [0,1)$ and as $n\to \infty$ we have $$\lim_{n\to \infty} p_\theta(x_{n+1}, x_n)=0$$
    Therefore by taking the limit $m,n \to \infty $ in inequality \eqref{Theo2: ineq1} we have $$\lim_{m,n\to \infty} p_\theta(x_{m}, x_n)=0$$
    which means the sequence $(x_n)$ is 0-Cauchy. Since $(X,p_\theta , \theta)$ is a $0$-complete, $(x_n)$ converges to a value called $u \in X$ such that $\lim_{n\to \infty} p_\theta(x_{n}, u)=p_\theta(u, u)=0$. Based on the triangle inequality of the partial extended b-metric, we obtain.
    \begin{equation} \label{Theo2: ineq 4}
        p_\theta(Tu,u) \leq \theta(Tu,u) [p_\theta(Tu, x_{n+1})+p_\theta(x_{n+1},u)]- p_\theta(x_{n+1},x_{n+1})
    \end{equation}
    From inequality $p_\theta (Tx,Ty) \leq k \left[ p_\theta(x,Tx) + p_\theta(y,Ty) \right]$ 
    setting $x=u$ and $y=x_{n}$ we have
    \begin{equation}\label{Theo2: ineq5}
        p_\theta (Tu,Tx_n) \leq k \left[ p_\theta(Tu,u) + p_\theta(x_n,x_{n+1}) \right] \quad , \forall m,n \in \mathbb{N}
    \end{equation}  
    Combine the inequality \eqref{Theo2: ineq 4} and \eqref{Theo2: ineq5} then we have.
    \begin{align*}
        p_\theta(Tu,u) \leq \theta(Tu,u) \cdot k \cdot p_\theta(Tu,u) + \theta(Tu,u) \cdot k \cdot p_\theta(x_{n+1}, x_n) \\ + \theta(Tu,u) p_\theta(x_{n+1},u) - p_\theta(x_{n+1},x_{n+1})
    \end{align*}
    So that for each $n\in \mathbb{N}$. 
    \begin{align*}
        [1- \theta(Tu,u)k] \cdot p_\theta(Tu,u) \leq \theta(Tu,u) \cdot k \cdot p_\theta(x_{n+1}, x_n) +\theta(Tu,u) \cdot \\ p_\theta (x_{n+1},u) - p_\theta(x_{n+1},x_{n+1})
    \end{align*}
    Therefore by limiting $n\to \infty $, we have $p_\theta(Tu,u)=0$ since $[1- \theta(Tu,u)k] \neq 0$. Consequently, $T$ has a fixed point. For the uniqueness let $w \in X$ is a fixed point of $T$. setting $x=u$ and $y=w$ in inequality $p_\theta(Tx,Ty) \leq k [p_\theta(x,Tx) + p_\theta(y,Ty)]$ yields $p_\theta(u,w) \leq k[p_\theta(u,u)+ p_\theta(w,w)]=k p_\theta(w,w)$. From the properties of a partial extended b-metric space, we have $p(w,w) \leq p_\theta(u,w)$. Consequently we obtain $p_\theta(u,w) \leq k p_\theta(w,w) \leq k p_\theta(u,w)$ then $p_\theta(u,w)=p_\theta(w,w)=p_\theta(w,w)=0$ since $k\in [0,\frac{1}{2})$, it shows that the fixed point $T$ is unique. 
\end{proof}
\vspace{5mm}

\begin{example}
    Let $X = [0, \infty)$. Define $p_{\theta} : X \times X \to [0, \infty)$ by $p_{\theta}(x,y) = \max \{x,y\}$, and let $\theta : X \times X \to [1, \infty)]$ be given by $\theta(x,y) = 1 + \frac{xy}{1 + x + y}$. Consider the mapping $T : X \to X$ defined by $Tx = \frac{x}{4}$. Show that T satisfies Theorem 3.10.
\end{example}
\vspace{0.5cm}
The following theorem is demonstrated using a modified version of the Kannan type condition introduced in the previous theorem.

\begin{theorem}
    Let $(X,p_\theta ,\theta)$ be a $0$-complete partial extended b-metric space. Let $T:X \to X$ be a mapping having the properties $\theta(Tx,x)< \frac{1}{k}$ and $$p_\theta(Tx,Ty)\leq k[p_\theta(x,Ty)+p_\theta(y,Ty)]$$ for all $x,y \in X$ where $k$ is constant in $[0, \frac{1}{2})$. If there exist an element $x_0 \in X$ such that $\lim_{n\to \infty }n \:p_\theta(T^nx_0,T^nx_0)=0$, Then $T$ has a unique fixed point $u \in X$. Moreover $u = \lim_{n\to \infty} T^nx_0$ and $p_\theta(u,u)=0$
\end{theorem}
    
\begin{proof}
    Construc the sequence $(x_n)_{n \geq0}$ with the initial term $x_0 \in X$ and recursion $x_n = Tx_{n-1}$ for all $n\in \mathbb{N}$. Setting $x= x_{m-1}$ and $y= x_{n-1}$ to the inequality $p_\theta(Tx,Ty)\leq k[p_\theta(x,Ty)+p_\theta(y,Ty)]$ over $m,n \in \mathbb{N}$ implies 
    \begin{equation}\label{Theo3 : 1}
        (\forall m,n \in \mathbb{N}) \quad p_\theta(x_m,x_n) -k \: p_\theta(x_{m-1},x_n) \leq k \: p_\theta(x_n ,x_{n-1})
    \end{equation}
    such that for all $m,n \in \mathbb{N}$:
    \begin{align*}
        p_\theta(x_{n+m},x_n)-k^mp_\theta(x_n,x_n) &= \sum_{t=1}^m \left[k^{t-1} p_\theta(x_{n+m-t+1}, x_n) - k^t \:p_\theta (x_{n+m+1},x_n)  \right]\\
        &\leq \left( \sum_{t=1}^m k^t \right) p_\theta(x_n,x_{n-1})\\
        &= \frac{1-k^{m+1}}{1-k} .p_\theta(x_n,x_{n-1})
    \end{align*}
    and we have
    \begin{equation}\label{Theo3 : 2}
        p_\theta(x_{n+m},x_n) \leq k^mp_\theta(x_n,x_n)+ \frac{1-k^{m+1}}{1-k} .p_\theta(x_n,x_{n-1}).
    \end{equation}
    Using \eqref{Theo3 : 1} and induction, the inequality holds for all $n\in \mathbb{N}$
    \begin{equation}\label{Theo3 : 3}
        p_\theta(x_{n+1},x_n) \leq k^n p_\theta(x_1,x_0) + k \sum_{t=1}^n k^{n-t}p_\theta(x_t,x_t).
    \end{equation}
    Observe that $\lim_{n\to \infty} p_\theta(T^nx_0,T^nx_0)=\lim_{n\to \infty} np_\theta(x_n,x_n)=0 $ such that $\lim_{n\to \infty} p_\theta(x_n,x_n)=0$. Also, observe that 
    \begin{equation}\label{Theo3 : 4}
        \sum_{t=1}^n k^{n-t}p_\theta(x_t,x_t) = \sum_{t=1}^n \frac{k^{n-t} t \:p_\theta(x_t,x_t)}{n+1}  + \sum_{t=1}^n \frac{k^{n-t} (n-t+1) \:p_\theta(x_t,x_t)}{n+1}
    \end{equation}
    A well-known theorem states that if a real sequence $(f_n)$ dan $(g_n)$ converge to $a$ and $b$ respectively, then
    $$\lim_{n \to \infty} \frac{1}{n+1} (f_1g_n +f_2g_{n-1}+ \cdots+f_ng_1)=ab.$$ Since $\lim_{n\to \infty} k^n =0$ and $\lim_{n \to \infty} np_\theta(x_n,x_n)=0$, then $$ \lim_{n\to \infty}\sum_{t=1}^n \frac{k^{n-t} t \:p_\theta(x_t,x_t)}{n+1}=0$$. Since $\lim_{n\to \infty} (n+1)k^n =0$ and $\lim_{n \to \infty} p_\theta(x_n,x_n)=0$, then $$ \lim_{n\to \infty}\sum_{t=1}^n \frac{k^{n-t} (n-t+1) \:p_\theta(x_t,x_t)}{n+1} =0.$$ 
    Limiting $n\to \infty$ to \eqref{Theo3 : 4} yields $$\lim_{n\to \infty }\sum_{t=1}^n k^{n-t}p_\theta(x_t,x_t)=0.$$
    Then by limiting $n \to \infty$ to \eqref{Theo3 : 3} yields $\lim_{n \to \infty}p_\theta(x_{n+1},x_n)=0$. Then, limiting $m,n \to \infty$ to \eqref{Theo3 : 2} we have 
    $$\lim_{n\to \infty} p_\theta(x_{m+n},x_n)=0.$$ Hence $(x_n)_{n \geq0}$ is a $0$-cauchy sequence. Since $(X,p_\theta ,\theta)$ is $0$-complete, then there exist $u\in X$ such that $\lim_{n\to \infty} p_\theta(x_n,u)=p_\theta(u,u)=0$. We will show that $u$ is a fixed point of $T$. According to triangle inequality of partial extended $b$-metric space, we have $(\forall n \in \mathbb{N})$
    \begin{equation} \label{theo3 : 5}
        p_\theta(Tu,u) \leq \theta(Tu,u)[p_\theta(Tu,x_{n+1})+p_\theta(x_{n+1},u)]- p_\theta (x_{n+1},x_{n+1}).
    \end{equation}
    Setting $x=u$ and $y=x_{n}$ to $p_\theta(Tx,Ty)\leq k[p_\theta(x,Ty)+p_\theta(y,Ty)]$ yields 
    \begin{equation}\label{theo3 : 6}
       (\forall n \in \mathbb{N}) \quad p_\theta(Tu,x_{n+1})\leq k[p_\theta(u,x_{n+1})+p_\theta(x_n,x_{n+1})]
    \end{equation}
    Combining \eqref{theo3 : 5} and \eqref{theo3 : 6}, we obtain $\forall n \in \mathbb{N}$: 
    \begin{align} \label{Theo3 : 7}
        (1- \theta(Tu,u)k) p_\theta(Tu,u) \leq \theta(Tu,u)k \cdot p_\theta(x_{n+1},x_n) + \theta(&Tu,u) p_\theta(x_{n+1},x_n) \nonumber \\ &- p_\theta(x_{n+1},x_{n+1})
    \end{align}
    
    Limiting $n \to \infty$ to \eqref{Theo3 : 7} yields $p_\theta(Tu,u)=0$ since $1-\theta(Tu,u)k \in (0,1]$. Hence $Tu=u$, i.e. $u$ is the fixed point of $T$. For the uniqueness of fixed point, let $w$ be a fixed point of $T$. Setting $(x,y) = (u,w)$ and $(x,y) = (w,u)$ to the inequality $p_\theta(Tx,Ty)\leq k[p_\theta(x,Ty)+p_\theta(y,Ty)]$, we will get $p_\theta(u,w) \leq \frac{k}{1-k} \cdot p_\theta(w,w)\leq p_\theta(w,w)$ and $p_\theta(u,w) \leq \frac{k}{1-k} \cdot p_\theta(u,u)\leq p_\theta(u,u)$. Since $p_\theta(x,x) \leq p_\theta(x,y)$ for all $x,y \in X$, then $p_\theta(u,w)=p_\theta(u,u)=p_\theta(w,w)$. Therefore $u=w$. In conclusion, $T$ has a unique fixed point $u$ where $p_\theta(u,u)=0$ and $u=\lim_{n \to \infty} T^n x_0$
\end{proof}

\begin{example}
    Let $X = [0,1]$. Define a function $p_{\theta}: X \times X \to \mathbb{R}$ by $p_{\theta}(x,y) = |x-y|+x, \forall x,y \in X$, and let $\theta : X \times X \to [1, \infty)$ be given by $\theta(x,y) = 1$. Define a mapping $T: X \to X$ by $T(x) = \frac{x}{4}, \forall x \in X$. Show that T satisfies Theorem 3.12.
\end{example}
\section{Conclusion}
As a logical extension of extended b-metric spaces (EBMS) and partial b-metric spaces, we have presented the idea of partial extended b-metric spaces (PEBMS) in this study. The suggested structure greatly expands the range of classical metric-type spaces by incorporating both the flexibility of a point-dependent control function $\theta$ and the allowance of non-zero self-distance. It has been demonstrated that when the self-distance is zero, any extended b-metric space can be thought of as a specific instance of a partial extended b-metric space. As a result, EBMS is a special subclass of the PEBMS class. This addition demonstrates the new structure's versatility and generality.

We also found a number of important results in this context, such as convergence properties, auxiliary lemmas, and fixed point theorems. We successfully extended classical results from EBMS to the PEBMS framework by introducing the concepts of 0-Cauchy sequences and 0-completeness. This shows that many important results in extended b-metric spaces still hold true in a broader context, as long as the right changes are made.In addition, the obtained fixed point results were applied to discrete dynamical systems, showing that PEBMS provides a suitable framework for analyzing stability and convergence behavior under non-standard distance structures. The presence of non-zero self-distance allows modeling systems with intrinsic uncertainty or internal state deviations.

In general, the idea behind PEBMS is that it creates a more flexible and complete mathematical structure that brings together and expands on several existing spaces. This work opens up new areas for future research, such as looking into more general contraction mappings, figuring out how stable nonlinear systems are, and possible uses in applied math and theoretical computer science. 

\section{Acknowledgment}
The author would like to express sincere gratitude to all authors whose works have contributed significantly to the development of this study. Their valuable ideas, results, and insights have provided a strong foundation for this research.

The author also extends appreciation to all individuals who have offered support, suggestions, and assistance, including those who contributed in small but meaningful ways throughout the completion of this work. Their help, encouragement, and attention to detail are deeply appreciated.



\begin{thebibliography}{99}

\bibitem{kamran2017}
Kamran, T., Samreen, M., UL Ain, Q., 2017,
A Generalization of $b$-Metric Space and Some Fixed Point Theorems,
\textit{Mathematics}, Vol. 5.

\bibitem{Kannan1968}
Kannan, R. (1968).
Some results on fixed points.
\textit{Bulletin of the Calcutta Mathematical Society}, Vol. 60 : 71--76.

\bibitem{Karahan2018PARTIALBP}
Karahan, I. and Isik, I. (2018).
Partial $b_{v}(s)$, partial $v$-generalized and $b_{v}(\theta)$ metric spaces and related fixed point theorems.
\textit{Facta Universitatis, Series: Mathematics and Informatics}, Vol. 35: 621--640

\bibitem{Bakhtin1989}
Bakhtin, I. A. (1989).
The contraction mapping principle in quasi-metric spaces.
\textit{Functional Analysis and Its Applications}, Vol. 30 : 26--37.

\bibitem{Czerwik1993}
Czerwik, S. (1993).
Contraction mappings in $b$-metric spaces.
\textit{Acta Mathematica et Informatica Universitatis Ostraviensis}, Vol. 1 : 5--11.

\bibitem{Bhattacharjee2025}
Bhattacharjee, K., Das, R., and Das, A. K. (2025).
New fixed point theorem on extended $b$-metric space using a class of contraction mappings.
\textit{Tatra Mountains Mathematical Publications}, Vol. 89 : 1--14.

\bibitem{Chen2020}
Chen, L., Li, C., Kaczmarek, R., and Zhao, Y. (2020).
Several fixed point theorems in convex $b$-metric spaces and applications.
\textit{Mathematics}, Vol. 8 : 242.

\bibitem{Dung2017}
Van Dung, N. and Le Hang, V. T. (2017).
Remarks on partial $b$-metric spaces and fixed point theorems.
\textit{Matematicki Vesnik}, Vol. 69 : 231--240.

\bibitem{Boriceanu2010}
Boriceanu, M., Bota, M., and Petrusel, A. (2010).
Multivalued fractals in $b$-metric spaces.
\textit{Open Mathematics}, Vol. 8 : 367--377.

\bibitem{Hussain2012}
Hussain, N., Dori\'c, D., Kadelburg, Z., Radenovi\'c, S., 2012,
Suzuki-Type Fixed Point Results in Metric Type Spaces,
\textit{Fixed Point Theory and Applications}, Vol. 2012, p. 126.

\bibitem{Matthews1994}
Matthews, S. G. (1994).
Partial metric topology.
\textit{Annals of the New York Academy of Sciences}, Vol. 728 : 183--197.

\bibitem{Shukla2014}
Shukla, S. (2014).
Partial $b$-metric spaces and fixed point theorems.
\textit{Mediterranean Journal of Mathematics}, Vol. 11 : 703--711.

\bibitem{Akkouchi2023}
Akkouchi, M. (2023).
On the Banach principle in $b$-metric spaces.
\textit{Bulletin of International Mathematical Virtual Institute}, Vol. 13.

\bibitem{agarwal2018banach}
Agarwal, P., Jleli, M., Samet, B., 2018,
Banach Contraction Principle and Applications,
in \textit{Fixed Point Theory in Metric Spaces: Recent Advances and Applications},
Springer, Singapore, pp. 1--23.

\end{thebibliography}





\end{document}